\documentclass[10pt]{article}

\usepackage{amsmath,amsthm,amssymb,fullpage,enumerate,epsfig}\usepackage[T1]{fontenc}
\usepackage[utf8]{inputenc}
\usepackage{authblk}

\usepackage{lipsum}

\newcommand\blfootnote[1]{%
  \begingroup
  \renewcommand\thefootnote{}\footnote{#1}%
  \addtocounter{footnote}{-1}%
  \endgroup
}

\newtheorem{theorem}{Theorem}
\newtheorem{lemma}[theorem]{Lemma}

\begin{document}

\title{The robber strikes back}
\author[1]{Anthony Bonato}
\author[2]{Stephen Finbow}
\author[3]{Przemys{\l}aw Gordinowicz}
\author[4]{Ali Haidar}
\author[1]{William B.\ Kinnersley}
\author[5]{Dieter Mitsche}
\author[1]{Pawe{\l} Pra{\l}at}
\author[6]{Ladislav Stacho}
\affil[1]{Ryerson University}
\affil[2]{St.\ Francis Xavier University}
\affil[3]{Technical University of Lodz}
\affil[4]{Carleton University}
\affil[5]{University of Nice Sophia-Antipolis}
\affil[6]{Simon Fraser University}

\date{}

\maketitle

\blfootnote{The authors were supported by grants from NSERC and Ryerson University.}

\begin{abstract}
We consider the new game of Cops and Attacking Robbers, which is identical to the usual Cops and Robbers game except that if the robber moves to a vertex containing a single cop, then that cop is removed from the game. We study the minimum number of cops needed to capture a robber on a graph $G$, written $cc(G)$. We give bounds on $cc(G)$ in terms of the cop number of $G$ in the classes of bipartite graphs and diameter two, $K_{1,m}$-free graphs.
\end{abstract}



\section{Introduction}

\emph{Cops and Robbers} is a vertex-pursuit game played on graphs that has
been the focus of much recent attention. Throughout, we only consider finite, connected, and simple undirected graphs.
There are two players consisting of
a set of \emph{cops} and a single \emph{robber}. The game is played over a
sequence of discrete time-steps or \emph{rounds}, with the cops going first
in the first round and then playing on alternate time-steps. The cops and
robber occupy vertices, and more than one cop may occupy a vertex.
 When a player is ready to move in a round they may move to a neighbouring vertex or \emph{pass} by remaining on their own vertex. Observe that any
subset of cops may move in a given round. The cops win if after some finite
number of rounds, one of them can occupy the same vertex as the robber. This
is called a \emph{capture}. The robber wins if he can avoid capture
indefinitely. A \emph{winning strategy for the cops} is a set of rules that
if followed result in a win for the cops, and a \emph{winning strategy for the
robber} is defined analogously.

If we place a cop at each vertex, then the cops are guaranteed to win.
Therefore, the minimum number of cops required to win in a graph $G$ is a
well defined positive integer, named the \emph{cop number} of the graph $G.$
We write $c(G)$ for the cop number of a graph $G$. For example, the Petersen graph has cop number $3$.
Nowakowski and Winkler~%
\cite{nw}, and independently Quilliot~\cite{q}, considered the game with one
cop only; the introduction of the cop number came in~\cite{af}. Many papers
have now been written on cop number since these three early works; see the
book \cite{bonato} for additional references and background on the cop
number. See also the surveys \cite{bonato1,bonato2,bonato3}.

Many variants of Cops and Robbers have been studied. For example, we may allow a cop to capture the robber from a distance $k$, where $k$ is a non-negative integer \cite{bonato4}, play on edges \cite{pawel}, allow one or both players to move with different speeds or teleport, or allow the robber to be invisible. See Chapter~8 of \cite{bonato} for a non-comprehensive survey of variants of Cops and Robber.

We consider a new variant of the game of Cops and Robbers, where the robber is able to essentially strike back against the cops. We say that the robber \emph{attacks} a cop if he chooses to move to a vertex on which a cop is present and eliminates her from the game.  In the game of \emph{Cops and Attacking Robbers}, the robber may attack a cop, but cannot start the game by moving to a vertex occupied by a cop; all other rules of the game are the same as in the classic Cops and Robbers. We note that if two cops are on a vertex $u$ and the robber moves to $u$, then only one cop on $u$ is eliminated; the remaining cop then captures the robber and the game ends. We write $cc(G)$ for the minimum number of cops needed to capture the robber. Note that $cc(G)$ is the analogue of the cop number in the game of Cops and Attacking Robber; our choice of notation will be made more transparent once we state Theorem~\ref{first}. We refer to $cc(G)$ as the \emph{cc-number of} $G.$ Since placing a cop on each vertex of $G$ results in a win for the cops, the parameter $cc(G)$ is well-defined.

To illustrate that $cc(G)$ can take different values from the cop number, consider that for the cycle $C_n$ with $n$ vertices we have the following equalities (which are easily verified):
$$cc(C_n)=
\begin{cases}
1 & \mbox{if } n=3,\\
2 & \mbox{if } 4\leq n \leq 6 ,\\
3 & \mbox{else.}
\end{cases} $$

We outline some basic results and bounds for the cc-number in Section~\ref{uno}. We consider bounds on $cc(G)$ in terms of $c(G)$ in Section~\ref{duo}. In Section~\ref{tri} we give the bound of $cc(G) \le c(G)+2$ in the case that $G$ is bipartite; see Theorem~\ref{bip}. In the final section, we supply in Theorem~\ref{cf} an upper bound for $cc(G)$ for $K_{1,m}$-free, diameter two graphs.

For background on graph theory see \cite{west}. For a vertex $u$, we let $N(u)$ denote the neighbour set of $u$, and $N[u]=N(u) \cup \{u\}$ denote the closed neighbour set of $u.$ The set of vertices of distance $2$ to $u$ is denoted by $N_2(u).$ We denote by $\delta(G)$ the minimum degree in $G$. In a graph $G,$ a set $S$ of vertices is a \emph{dominating set} if every vertex not in $S$ has a neighbor in $S.$ The \emph{domination number} of $G,$ written $\gamma (G),$ is the minimum cardinality of a dominating set. The \emph{girth} of a graph is the length of the shortest cycle contained in that graph, and is $\infty$ if the graph contains no cycles.

\section{Basic results}\label{uno}

In this section we collect together some basic results for the cc-number. As the proofs are either elementary or minor variations of the analogous proofs for the cop number, they are omitted. The first result on the game of Cops and Attacking Robbers is the following theorem; note that the second inequality naturally inspires the notation $cc(G)$. We use the notation $\bar{c}(G)$ for the edge cop number, which is a variant where the cops and robber move on edges; see \cite{pawel}.
\begin{theorem}\label{first}
If $G$ is a graph, then $$c(G) \leq cc(G) \leq \min\{2 c(G),2 \bar{c}(G), \gamma(G)\}.$$
\end{theorem}

The following theorem is foundational in the theory of the cop number.

\begin{theorem}\label{girth0}\cite{af}
If  $G$ has girth at least 5, then $$c(G) \geq \delta (G).$$
\end{theorem}
The following theorem extends this result to the cc-number.
\begin{theorem}\label{girth}
If  $G$ has girth at least 5, then $$cc(G) \geq \delta (G)+1.$$
\end{theorem}

Isometric paths play an important role in several key theorems in the game of Cops and Robbers, such as the cop number of planar graphs (see Chapter 4 of \cite{bonato}). We call a path $P$ in a graph $G$ \emph{isometric} if the shortest distance between any two vertices is equal in the graph induced by $P$ and in $G$. For a fixed integer $k \geq 1$, an induced subgraph $H$ of $G$ is $k$-$guardable$ if, after finitely many moves, $k$ cops can move only in the vertices of $H$ in such a way that if the robber moves into $H$ at round $t$, then he will be captured at round $t + 1$ by a cop in $H.$ For example, a clique in a graph is 1-guardable.

Aigner and Fromme \cite{af} proved the following result.

\begin{theorem}\cite{af}\label{path0}
An isometric path is 1-guardable.
\end{theorem}

We have an analogue of Theorem~\ref{path0} for the cc-number.

\begin{theorem}\label{path}
An isometric path is 2-guardable in the game of Cops and Attacking Robbers, but need not be 1-guardable.
\end{theorem}
See Figure~\ref{fig1} for an example where the robber can freely move onto an isometric path without being captured by a sole cop.
\begin{figure}[h!]
\begin{center}
\epsfig{figure=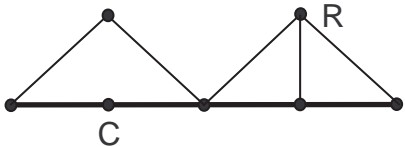}\caption{One cop cannot guard the isometric path (depicted in bold). We assume that the robber has just arrived at their vertex and it is the cop's turn to move.}
\label{fig1}
\end{center}
\end{figure}

A graph $G$ is called \emph{planar} if it can be embedded in a plane without two of its edges crossing. It was shown first in \cite{af} that planar graphs require at most three cops to catch the robber; see \cite{bonato} for an alternative proof of this fact. Given the results above, we may conjecture that the cc-number of a planar graph is at most 4 or even 5, but either bound remains unproven.

\emph{Outerplanar} graphs are those that can be embedded in the plane without crossings in such a way that all of the vertices belong to the unbounded face of the embedding. Clarke proved the following theorem in her doctoral thesis.
\begin{theorem}\cite{nc}\label{outer0}
If $G$ is outerplanar, then $c(G) \leq 2$.
\end{theorem}

The counterpart to Theorem~\ref{outer0} is the following.

\begin{theorem}\label{outer}
If $G$ is outerplanar, then $cc(G) \leq 3$.
\end{theorem}

Meyniel's conjecture---first communicated by Frankl~\cite{f}---is one of the most important open problems surrounding the game of Cops and Robbers. The conjecture states that $c(n) = O(\sqrt{n})$,  where $c(n)$ is the maximum of $c(G)$ over all $n$-vertex, connected graphs. Cops and Robbers has been studied extensively for random graphs (see for example, \cite{bkl, bpw, lp, p}), partly owing to a search for counterexamples to Meyniel's conjecture. However, it was recently shown that Meyniel's conjecture holds asymptotically almost surely (that is, with probability tending to $1$ as the number of vertices tends to infinity) for both binomial random graphs $G(n,p)$~\cite{pw1} as well as random $d$-regular graphs~\cite{pw2}.

In \cite{bpw} it was shown that for dense random graphs, where $p=n^{-o(1)}$ and $p<1-\epsilon$ for some $\epsilon>0$, asymptotically almost surely we have that
\begin{equation}\label{rand}
c(G(n,p)) = (1+o(1)) \gamma(G(n,p)) = (1+o(1)) \log_{1/(1-p)} n,
\end{equation}
Note that (\ref{rand}) implies that $c(G(n,p)) = (1+o(1)) cc(G(n,p))$ for the stated range of $p$; in particular, applying (\ref{rand}) to the $p=1/2$ case (which corresponds to the uniform probability space of all labelled graphs on $n$ vertices), we have that for every $\epsilon > 0$, almost all graphs satisfy $cc(G)/c(G) \in [1,1+\epsilon]$. Unfortunately, the asymptotic value of the cop number is not known for sparser graphs. However, it may be provable that $c(G(n,p)) = (1+o(1)) cc(G(n,p))$ for sparse graphs, without finding an asymptotic value.

We finish the section by noting that graphs with $cc(G)=1$ are precisely those with a universal vertex. However, characterizing those graphs $G$ with $cc(G)=2$ is an open problem. Graphs with $cc(G)=2$ include cop-win graphs without universal vertices, and graphs which are not cop-win but have domination number 2.  Before the reader conjectures this gives a characterization, note that the graph in Figure~\ref{eg} with cc-number equaling 2 is in neither class.
\begin{figure}[h!]
\begin{center}
\epsfig{figure=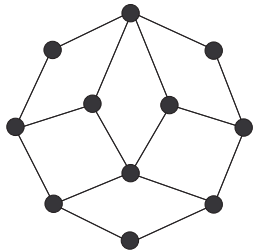,height=1.5in,width=1.5in}\caption{A graph $G$ with $c(G)=cc(G)=2$ and $\gamma(G)=3$.}
\label{eg}
\end{center}
\end{figure}

\section{How large can the cc-number be?}\label{duo}

One of the main unanswered questions on the game of Cops and Attacking Robbers is how large the cc-number can be relative to the cop number. Many of the results from the last section might lead one to (mistakenly) conjecture that $$cc(G) \le c(G) +1$$ for all graphs, and this was the thinking of the authors and others for some time. We provide a counterexample below.

By Theorem~\ref{first}, we know that $cc(G)$ is bounded above by $2c(G).$ For example, this is a tight bound for a path of length at least $3$. However, we do not know an improved bound which applies to general graphs, nor do we possess graphs $G$ with $c(G)>2$ whose cc-number equals $2c(G)$. In this section, we outline one approach which may ultimately yield such examples. Improved bounds for several graph classes are outlined in the next two sections.

Our construction utilizes line graphs of hypergraphs. For a positive integer $k$, a $k$-\emph{uniform hypergraph} has every hyperedge of cardinality $k$. A hypergraph is \emph{linear} if any two hyperedges intersect in at most one vertex. The \emph{line graph} of a hypergraph $H$, written $L(H)$, has one vertex for each hyperedge of $H,$ with two vertices adjacent if the corresponding hyperedges intersect.

\begin{lemma}\label{lem:hyper_lower}
Let $H$ be a linear $k$-uniform hypergraph with minimum degree at least 3 and girth at least 5.  If $L(H)$ has domination number at least $2k$, then $cc(L(H)) \ge 2k$.
\end{lemma}
\begin{proof}
Suppose there are at most $2k-1$ cops. Since the domination number of $L(H)$ is at least $2k,$ the robber can choose an initial position that lets him survive the cops' first move. To show that $2k-1$ cops cannot catch the robber in the game of Cops and Attacking Robber on $L(H)$, suppose otherwise, and consider the state of the game on the robber's final turn (that is, just before he is to be captured).   Let $v$ be the robber's current vertex, $E_v$ the corresponding edge of $H$, and $w_1, w_2, \ldots, w_k$ the elements of $E_v$.  The neighbours of $v$ in $L(H)$ are precisely those vertices corresponding to edges of $H$ that intersect $E_v$; denote by $S_{w_i}$ the set of vertices (other than $v$) corresponding to edges containing $w_i$.  Each $S_{w_i}$ is a clique; moreover, since $H$ has minimum degree at least 3, each contains at least two vertices.  By hypotheses for $H$, it follows that the $S_{w_i}$ are disjoint and that no vertex outside $S_{w_i}$ dominates more than one vertex inside.  Finally, since $H$ has girth at least 5, no vertex in $G$ dominates vertices in two different $S_{w_i}$ (that is, the neighbourhoods $N[S_{w_i}]$ only have $v$ in common).

Consider the cops' current positions.  The cops must dominate all of $N[v]$, since otherwise the robber would be able to survive for one more round (by moving to an undominated vertex).  Since the $N[S_{w_i}]$ only have $v$ in common, for some $j$ we have at most one cop in $N[S_{w_j}]$.  If in fact there are no cops in $N[S_{w_j}]$, then no vertices of $S_{w_j}$ are dominated, a contradiction.  Thus, $S_{w_j}$ contains exactly one cop.  Since each vertex outside $S_{w_j}$ dominates at most one vertex inside and $S_{w_j}$ contains at least two vertices, the cop must actually stand within $S_{w_j}$.  However, since she is the only cop within $N[S_{w_j}]$, the robber may attack the cop without leaving himself open to capture on the next turn.  Thus, the robber always has a means to avoid capture on the cops' next turn.  Hence, at least $2k$ cops are needed to capture the robber, as claimed.
\end{proof}

We aim to find, for all $k$, graphs $G$ such that $c(G)=k$ and $cc(G)=2k$. This, however, remains open for all $k\ge 3.$

As an application of the lemma, take $H$ to be the Petersen graph.  It is easily verified that $c(L(H))=2$; see also \cite{pawel}. Lemma~\ref{lem:hyper_lower} with $k=2$ shows that $cc(L(H)) \ge 4$; hence, Theorem~\ref{first} then implies that $cc(L(H))=4$. See Figure~\ref{fig2} for a drawing of the line graph of the Petersen graph.
\begin{figure}[h!]
\begin{center}
\epsfig{figure=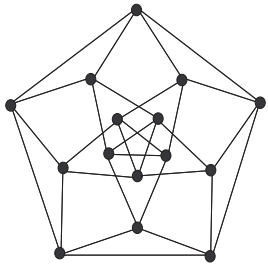,height=1.5in,width=1.5in}\caption{The line graph of the Petersen graph.}\label{fig2}
\end{center}
\end{figure}

\section{Bipartite graphs}\label{tri}

For bipartite graphs, we derive the following upper bound.

\begin{theorem}\label{bip}
For every connected bipartite graph $G$, we have that $cc(G) \le c(G)+2$.
\end{theorem}

\begin{proof}
Fix a connected bipartite graph $G$.  Let $k=c(G)$; we give a strategy for $k+2$ cops to win the game of Cops and Attacking Robbers on $G$.  Label the cops $C_1, C_2, \ldots, C_k, C^*_1, C^*_2$.  Intuitively, cops $C_1, C_2, \ldots, C_k$ attempt to follow a winning strategy for the ordinary Cops and Robber game on $G$; since they must avoid being killed by the robber, they may not be able to follow this strategy exactly, but can follow it ``closely enough''.  Cops $C^*_1$ and $C^*_2$ play a different role: they occupy a common vertex throughout the game, and in each round, they simply move closer to the robber.  This has the effect of eventually forcing the robber to move on every turn.  (Since the cops move together, the robber cannot safely attack either one.)  Further, when the robber passes, the cops $C_1, C_2, \ldots, C_k, C^*_1, C^*_2$ pass. Therefore, we may suppose throughout that the robber moves to a new vertex on each turn.  

It remains to formally specify the movements of $C_1, C_2, \ldots, C_k$.  To each cop $C_i$, we associate a \emph{shadow} $S_i$.  Throughout the game the shadows follow a winning strategy for the ordinary game on $G$.  Let $C_i^{(t)}$, $S_i^{(t)}$, and $R^{(t)}$ denote the positions of $C_i$, $S_i$, and the robber, respectively, at the end of round $t$.  We maintain the following invariants for $1 \le i \le k$ and all $t$:
\begin{enumerate}
\item $S_i^{(t)} \in N[C_i^{(t)}]$ (that is, each cop remains on or adjacent to her shadow);
\item if $C_i^{(t+1)} \not = S_i^{(t+1)}$, then $S_i^{(t+1)}$ and $R^{(t)}$ belong to different partite sets of $G$;
\item $C_i^{(t+1)}$ is not adjacent to $R^{(t)}$ (that is, the robber never has the opportunity to attack any cop).
\end{enumerate}
On round $t+1$, each cop $C_i$ moves as follows:
\begin{itemize}
\item[(a)] If $C_i^{(t)} \not = S_i^{(t)}$, then $C_i$ moves to $S_i^{(t)}$;
\item[(b)] if $C_i^{(t)} = S_i^{(t)}$, and $S_i^{(t+1)}$ is not adjacent to $R^{(t)}$, then $C_i$ moves to $S_i^{(t+1)}$;
\item[(c)] otherwise, $C_i$ remains at her current vertex.
\end{itemize}
By invariant (1), this is clearly a legal strategy.

We claim that all three invariants are maintained.  Invariant (1) is straightforward to verify.  For invariant (2), first suppose that $C_i^{(t)}=S_i^{(t)}$, but $C_i^{(t+1)} \not = S_i^{(t+1)}$.  By the cops' strategy, this can happen only when $S_i^{(t+1)}$ is adjacent to $R^{(t)}$, in which case the shadow and robber belong to different partite sets, as desired.  Now suppose that $C_i^{(t)} \not = S_i^{(t)}$ and $C_i^{(t+1)} \not = S_i^{(t+1)}$.  By the cops' strategy we have $C_i^{(t+1)} = S_i^{(t)}$.  It follows that $C_i^{(t+1)} \not = C_i^{(t)}$, $S_i^{(t+1)} \not = S_i^{(t)}$, and $R^{(t-1)} \not = R^{(t)}$.  Thus, if $S_i^{(t)}$ and $R^{(t-1)}$ belong to different partite sets, then so must $S_i^{(t+1)}$ and $R^{(t)}$; that is, the invariant is maintained.  For invariant (3), if $S_i^{(t+1)}$ is adjacent to $R^{(t)}$, then we may suppose that $S_i^{(t+1)} \not = S_i^{(t)}$, since otherwise the shadow would have captured the robber in round $t+1$.  By the cops' strategy, we now have that $C_i^{(t+1)} \not = S_i^{(t+1)}$.  But now the cop and her shadow are in different partite sets by invariant (1), and the shadow and robber are in different partite sets by invariant (2), so the cop and robber are in the same partite set, contradicting adjacency of the cop and the robber.

Since the shadows follow a winning strategy, eventually some shadow $S_i$ captures the robber; that is, for some $t$, we have that either $S_i^{(t)}=R^{(t)}$ or $S_i^{(t+1)} = R^{(t)}$. In the former case, invariant (3) implies that $C_i^{(t)}\ne S_i^{(t)}$ and invariant (1) implies that $C_i$ captures the robber in round $t+1.$ Now consider that case when $S_i^{(t+1)} = R^{(t)}$.  By invariant (2), since $S_i^{(t+1)}$ is not adjacent to $R^{(t)}$, we in fact have that $C_i^{(t+1)} = S_i^{(t+1)} = R^{(t)}$, so the cops have won.
\end{proof}

\section{$K_{1,m}$-free, diameter $2$ graphs}

We provide one more result giving an upper bound on the cc-number for a set of graph classes.

\begin{theorem}\label{cf}
Let $G$ be a $K_{1,m}$-free, diameter $2$  graph, where $m\geq 3.$ Then
\[
cc(G)\leq c(G)+2m-2.
\]
\end{theorem}

When $m=3$, Theorem~\ref{cf} applies to claw-free graphs; see \cite{cs} for a characterization of these graphs. The cop number of diameter 2 graphs was studied in \cite{bb}.

\begin{proof}[Proof of Theorem~\ref{cf}]
A cop $C$ is \emph{back-up} to a cop $C^{\prime }$ if $C$ is in $N[C^{\prime
}].$ Note that a cop with
a back-up cannot be attacked without the robber being captured in the next
round.

Now let $c(G)=r,$ and consider $c(G)$ cops labelled $C_{1},C_{2},\ldots
,C_{r}.$\ We refer to these $r$-many cops as \emph{squad} $1.$ Label an
additional $2m-2$ cops as $\widehat{C_{i,1}}$ and $\widehat{C_{i,2}},$ where
$1\leq i\leq m-1;$ these cops form \emph{squad} $2.$ The intuition behind
the proof is that the cops in squad $2$ act as back-up for those in squad
1$,$ who play their usual strategy on $G.$ Further, the cops $\widehat{
C_{i,j}}$ are positioned in such a way that the cops $C_{k}$ need only
restrict their movements to the second neighbourhood of some fixed vertex.

More explicitly, fix a vertex $x$ of $G.$ Move squad $2$ so that they are
contained in $N[x].$ Next, position each of the cops $\widehat{C_{i,1}}$ on $
x.$ Hence, $R$ must remain in $N_{2}(x)$ or he will lose in the next round (in
particular, no squad 2 cop is ever attacked). Throughout the game we will
always maintain the property that there are $m-1$ cops on $x.$

We note that the squad 2 cops in $N(x)$ can move there essentially as if
that subgraph were a clique, and in addition, preserve the property that $m-1
$ cops remain on $x.$ To see this, if $\widehat{C_{i,2}}$ were on $y\in N(x)$
and the cops would like to move to $z\in N(x),$ then move $\widehat{C_{i,2}}$
to $x,$ and move some squad 2 cop  from $x$ to $z.$ In particular, a cop
from squad 2 can arrange things so that she is adjacent to a cop in squad 1
after at most one move.  We refer to this movement of the squad two cops as
a \emph{hop}, as the cops appear to jump from one vertex of $N(x)$ to
another (although what is really happening is that the cops are cycling
through $x).$ Note that hops maintain $m-1$ cops on $x$.

We now describe a strategy $\mathcal{S}$ for the cops, and then show that it
is winning. The cops in squad 1 play exactly as in the usual game of Cops
and Robbers; note that the squad 1 cops may leave $N_{2}(x)$ depending on
their strategy, but $R$ will never leave $N_{2}(x).$ The squad 2 cops play
as follows. Squad 2 cops do not move unless the following occurs:\ a squad 1
cop $C_{k}$ moves to a neighbour of $R,$ and $C_k$ has no back-up from a squad 1 cop.
In that case, some squad 2 cop $
\widehat{C_{i,j}}$ hops to a vertex of $N(x)$ which is adjacent to $C_{k}.$
There are a sufficient number of squad 2 cops to ensure this property, since
if $m$ (or more) squad 1 cops move to neighbours of $R,$ then some of these
cops must be adjacent to each other as $G$ is $K_{1,m}$-free (in particular,
the cops in $N(R)$ play the role of back-ups to each other).

Hence, the squad 1 cops may apply their winning strategy in the usual game
and ensure that whenever they move to a neighbour of $R$, some squad 2 cop
serves as back-up. In particular, $R$ will never attack a squad 1 cop for
the duration of the game. Thus, $\mathcal{S}$ is a winning strategy in the game of Cops and Attacking Robbers.
\end{proof}

\end{document}